\documentclass{amsart}
\usepackage{amssymb,mathrsfs}
\usepackage{color,umoline}
\usepackage[dvipsnames]{xcolor}
\usepackage{graphicx}
\usepackage{enumitem}
\usepackage[font=small,labelfont=bf]{caption}
\usepackage{tikz, caption}
\usepackage{relsize}
\usepackage[mathscr]{eucal}
\usepackage{dsfont}
\usepackage{verbatim}
\usepackage{kotex}
\usepackage[draft]{todonotes}
\usepackage[
citecolor=black,  
hypertexnames=false]{hyperref}
\usepackage[dvipsnames]{xcolor}
\usetikzlibrary{arrows.meta}
\usepackage{subfig}

\newcommand{\N}{\mathbb{N}}
\newcommand{\R}{\mathbb{R}}
\newcommand{\Z}{\mathbb{Z}}

\newcommand{\re}{\mathop{\mathrm{Re}}}
\newcommand{\im}{\mathop{\mathrm{Im}}}

\newcommand{\dist}{\operatorname{dist}}

\newtheorem{thm}{Theorem}[section]

\newtheorem{prop}[thm]{Proposition}
\newtheorem{cor}[thm]{Corollary}
\newtheorem{lem}[thm]{Lemma}
\numberwithin{equation}{section}
\newtheorem{rem}{Remark}

\newcommand\CI{{\mathcal I}}

\newcommand{\inp}[2]{\langle #1, #2\rangle}

\newcommand{\ol}[1]{\overline{#1}}

\newcommand{\pq}{(\tfrac1p, \tfrac1q)}
\newcommand{\ppq}{(1/p,1/q)}
\newcommand{\Be}{\begin{equation}}
\newcommand{\Ee}{\end{equation}}

\newcommand{\qs}{q_\ast}
\newcommand{\ps}{p_\ast}

\newcommand{\sucp}{{\bf sucp\! }}
\begin{document}

\author[Jeong]{Eunhee Jeong}
\address[Jeong]{Department of Mathematics Education, and  Institute of Pure and Applied Mathematics, Jeonbuk National University, Jeonju 54896, Republic of Korea}
\email{eunhee@jbnu.ac.kr}

\author[Lee]{Sanghyuk Lee}
\address[Lee]{Department of Mathematical Sciences and RIM, Seoul National University, Seoul 08826, Republic of  Korea}
\email{shklee@snu.ac.kr}

\author[Ryu]{Jaehyeon Ryu}
\address[Ryu]{School of Mathematics, Korea Institute for Advanced Study, Seoul 02455, Republic of Korea} 
\email{jhryu@kias.re.kr}

\title[Strong unique continuation for the heat operator]
{Unique continuation for the heat operator 
\\\ with potentials in 
weak spaces
}

\begin{abstract} 
We prove strong unique continuation property for the differential inequality $|(\partial_t +\Delta)u(x,t)|\le V(x,t)|u(x,t)|$ with $V$ contained in
weak spaces. In particular, we establish the strong unique continuation property for 
$V\in L^\infty_t L^{d/2,\infty}_x$, which has been left open  since the works of Escauriaza \cite{e00} and Escauriaza-Vega \cite{ev01}. 
Our results are consequences of the Carleman estimates for the heat operator in the Lorentz spaces.  

\end{abstract}

\maketitle

\section{Introduction}

We  consider the differential inequality  
\begin{equation}\label{diff_ineq}
 |(\partial_t +\Delta)u(x,t)|\le |V(x,t)u(x,t)|,\quad (x,t)\in \R^d\times (0,T). 
 \end{equation}
For a differential operator $P$ on a domain $\Omega$, the strong unique continuation property (abbreviated  to  {\bf sucp} in what follows) for $|Pu|\le |Vu|$ means that a nontrivial solution $u$ to $|Pu|\le |Vu|$ cannot vanish to infinite order (in a suitable sense) at any point. 
The \sucp for  second order parabolic operator  has been studied by many authors (see \cite{lin, so90, poon, chen,e00,ev01,ef03, Fe03,T09, BM} 
  and references therein). In particular,  the study of \sucp  for the heat operator with time-dependent potentials  goes back to  Poon \cite{poon} and Chen \cite{chen}, who 
  considered bounded potentials.  Escauriaza \cite{e00} and Escauriaza and Vega \cite{ev01}  extended the results   to unbounded potentials $V$  under the  parabolic vanishing condition: That is to say,  for given $\delta>0$ and 
any $k\in \mathbb N$  there is a constant $C_k$ such that 
\begin{equation}\label{parabolic-decay}
 |u(x,t)|\le C_k (|x| + \sqrt{t})^k e^{(1-\delta)|x|^2/8t},\quad (x,t)\in \R^d\times (0,T).
 \end{equation}
Here, the growth condition at infinity is necessary since there exists a nonzero solution $u$ to $(\partial_t +\Delta)u=0$ such that $u$ vanishes to infinite order in the space-time variables at any point $(x,0)$, $x\in \R^d$ (see, for example, \cite{john,e00}).

The \sucp for the Laplacian $-\Delta$
is better understood. Since the pioneering work of Carleman \cite{Ca39},  most of subsequent results were obtained by following his idea, the Carleman weighted inequality. In particular, Jerison and Kenig \cite{JK} proved the \sucp for the Laplacian with $V\in L^{d/2}_{loc},$ $d\ge3$.  
Their result  was  extended by Stein \cite{St85} to potentials  $V\in  L^{d/2,\infty}$ under the assumption that  $\|V\|_{L^{d/2,\infty}}$ is small enough. 
Later,  Wolff \cite{W92} showed that the smallness assumption is indispensable if  $V\in L^{ d/2,\infty}$. 
Here,  $\|\cdot\|_{p,r}$ denotes the  norm of  the  Lorentz space $L^{p,r}(\mathbb R^d)$ (for example, see \cite{St71}).

By the aforementioned results due to Escauriaza \cite{e00} and Escauriaza and Vega \cite{ev01} the \sucp for \eqref{diff_ineq} is known when 
$t^{1-d/(2\mathfrak r)-1/\mathfrak s}V(x,t)\in L_{t,loc}^{\mathfrak s}L_x^{\mathfrak r}$ and $\mathfrak r,\,\mathfrak s$ satisfy 
 \Be 
 \label{rs-con}
 d/(2\mathfrak r)+1/\mathfrak s\le  1, \quad  1 \le \mathfrak r, \mathfrak s \le\infty. 
 \Ee
However, in view of those results concerning  the (abovementioned) \sucp for the Laplacian,    it seems  natural to  expect that  the class of potentials for which the \sucp for \eqref{diff_ineq} holds  can be further expanded to certain weak spaces.  
 
In this paper, we extend the results in  \cite{e00, ev01} to a  larger class of potentials,  that is to say, $t^{1-d/(2 \mathfrak r)-1/\mathfrak s}V(x,t)\in L_{t,loc}^{\mathfrak s}L_x^{\mathfrak r,\infty}$, $d/(2\mathfrak r)+1/\mathfrak s\le  1$, $d/2\le\mathfrak  r \le\infty$.  
As in \cite{St85} our result is a consequence of  new Carleman estimates for the heat operator in the Lorentz spaces. 

\subsubsection*{Carleman estimate} 
We denote $L^s_tL_x^{q,b}=L^s_t(\R_+; L_x^{q,b}(\mathbb R^d))$.  Then, we consider the Carleman  inequality for the heat operator of the form
\begin{equation}\label{carl}
\begin{aligned}
\Big\| &t^{-\alpha}e^{-\frac{|x|^2}{8t}}g\Big\|_{L^s_tL_x^{q,b}} \le  C\Big\|t^{-\alpha +1-\frac d2(\frac1p-\frac1q)-(\frac1r-\frac1s)} e^{-\frac{|x|^2}{8t}}(\Delta +\partial_t) g\Big\|_{L^r_t L_x^{p,a}}
\end{aligned}
\end{equation}
with $C$ independent of $\alpha$, which holds  for  $g\in \mathrm C_c^\infty(\mathbb R^{d+1}\setminus\{(0,0)\})$  under a suitable condition on the exponents $\alpha, p, q,r,s, a$ and $b$.
For  $\alpha \in \mathbb R$, we set \[ \beta= \beta(\alpha, q,s)=2\alpha -\frac d{q}-\frac 2s.\]
The estimate \eqref{carl} was formerly considered with  $p=a,q=b$.  It was Escauriaza \cite{e00} who first obtained the estimate \eqref{carl} for some $p=a,q=b,r,s$.  
 More precisely,  he showed  that the estimate \eqref{carl} holds with  the Lebesgue spaces (i.e., $a=p$, $b=q$) for $p,q$ satisfying $q=p'$ and $1/p-1/q<2/d$ if $d\ge2$, and $1/p-1/q\le 1$ if $d=1$  provided that   
\[
 \dist(\beta, \mathbb N_0)\ge c
\]
  for some $c>0$ where $ \N_0  := \N\cup \{0\}$.  Subsequently, the estimate \eqref{carl} was  extended  by Escauriaza and Vega \cite{ev01} to the exponents $p,q$ which lie outside of the line of duality. They obtained the estimate \eqref{carl} for ${2d}/{(d+2)}\le p\le 2\le q\le {2d}/{(d-2)}$  if $d\ge3$,   and  for $1\le p\le 2\le q\le\infty$, $(p,q,d)\neq (1,\infty,2)$ for $d=1,2$.

We extend the previously known results not only to  Lorentz spaces but also on a wider range of exponents $p,q,r,s$.  To present our result, for $d\ge 3$ we define $\mathscr A=  \mathscr A(d)$, $\mathscr B=  \mathscr B(d)$, $\mathscr D=  \mathscr D(d)\in [1/2,1]\times [0,1/2]$ by setting 
\begin{align*}
\mathscr A&=\left(\frac{d+2}{2d},\frac 12\right),\; 
\quad \mathscr B=\left(\frac{d^2+2d-4}{2d(d-1)},\frac{d-2}{2(d-1)}\right),  \quad  \mathscr D=\left(\frac{d+2}{2d},\frac{d-2}{2d}\right).
\end{align*}
By  $\mathfrak T$ we denote the closed pentagon with vertices $(\frac12,\frac12)$, $\mathscr A$, $\mathscr B$, $\mathscr B'$, and $\mathscr A'$ from which the two vertices $\mathscr B$ and $\mathscr B'$ are removed. Here,  $X'=(1-b,1-a)$  (the dual point)  if $X=(a,b)$.   See Figure \ref{fig:uniform}.

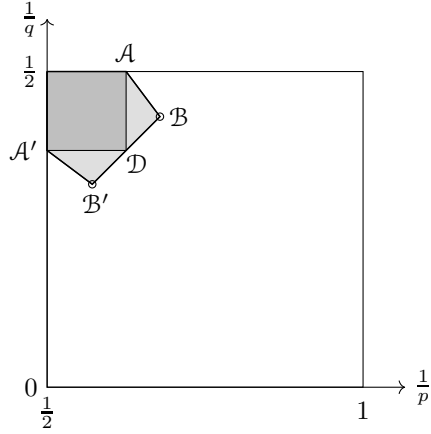
\begin{figure}[t]
\centering
\begin{tikzpicture}[scale=0.7]
\draw[<->] (0,7) node[left]{$\frac1q$}--(0,0)--(6.8,0)node[right]{$\frac1p$};
\draw (0,0) rectangle (6,6);
\filldraw[fill=gray!25](0,6)--(0,4.5)--(6/7,27/7) --(15/7,36/7)--(1.5,6);
\filldraw[fill=gray!50](0,6)--(0,4.5)--(1.5, 4.5)--(1.5,6);
\draw [line width=0.2mm]  (0,6)--(0,4.5)--(6/7,27/7) --(15/7,36/7)--(1.5,6)--(0,6); 
\node[below] at (1.7, 4.6) {$\mathscr D$};
\node[left] at (0,4.5) {$\mathscr A'$}; \node[above] at (1.5,6) {$ \mathscr A$}; 
\draw (6/7,27/7) circle [radius=0.07]; \node[below] at (6.5/7,27/7) {$\mathscr B'$};
\draw (15/7,36/7) circle [radius=0.07]; \node[right] at (15/7,36/7) {$\mathscr B$}; 
\node[left] at (0,0) {$0$};\node[below] at (0,0) {$\frac12$};
\node[left] at (0,6) {$\frac12$};\node[below] at (6,-0.1) {$1$};
\end{tikzpicture}
\caption{The region of $p,q$ for which  \eqref{carl} holds when $d\ge 3$: The dark grayed square stands for 
the earlier result due to  Escauriaza and Vega \cite{ev01}, and  the light grayed region for the newly extended range. 
}
\label{fig:uniform}
\end{figure}

\begin{thm}\label{thm_carl}
Let $d\ge 3$ and $(1/p,1/q)\in \mathfrak T$. Let $1\le r\le s\le\infty$ satisfy $(\frac1r,\frac1s)\neq (1,\frac{d}2(\frac1p-\frac1q)),$ $(1-\frac d2(\frac1p-\frac1q),0)$, and
\begin{equation}\label{ca_co2}
0\le \frac1r-\frac 1s \le 1-\frac d2\big(\frac1p-\frac1q\big).
\end{equation} 
Suppose $\beta \notin \mathbb N_0$. Then,  if $\frac1p-\frac1q<\frac2d$, $p\neq 2$, and $q\neq 2$,   the  estimate \eqref{carl} 
 holds for $1\le a=b\le \infty$ with $C$ depending only on $p,q,a,b, r,s,$ and $\dist(\beta,\mathbb N_0)$$;$  if 
 \begin{equation}\label{critical}
\frac1p-\frac1q=\frac2d, 
\end{equation} 
the  same estimate \eqref{carl} 
 holds for $ a=b=2$.   
\end{thm}

It is remarkable that  Theorem  \ref{thm_carl} gives \eqref{carl} for $(1/p,1/q)$ contained in the open line segment $(\mathscr B, \mathscr B')$ (see Figure \ref{fig:uniform}).   
The exponents $p,q$ satisfying  \eqref{critical} 
constitute the critical case 
in 
that \eqref{carl} is no longer true if $1/p-1/q>2/d$.   (See Remark \ref{failure} and the condition \eqref{ca_co2}.) 
Consequently, it is more difficult  to obtain the estimate \eqref{carl} for $p,q$ satisfying \eqref{critical} than that  for $p,q$ satifying $1/p-1/q<2/d$. 
Only the estimate \eqref{carl} with $(1/p,1/q)=\mathscr D$, $a=p$, and $b=q$ was previously shown by Escauriaza and Vega \cite{ev01}. 

If  $p,q$ satisfy \eqref{critical} and $r=s$, then   the  estimate \eqref{carl}   implies the Carleman inequality for the Laplacian (see \cite{ev01}):
\begin{equation}\label{carl_el} \| |x|^{-\sigma} f\|_{L^{q,b}} \le C \||x|^{-\sigma} \Delta f\|_{L^{p,a}},\quad f\in \mathrm C_c^\infty(\R^d\setminus \{0\})\end{equation}
for $\sigma>0$ with $C>0$ depending on $d, p,q$ and $\dist(\sigma, \mathbb N+\frac dq)$ if $\dist(\sigma, \mathbb N+\frac dq)>0.$   
By this implication  the estimates in Theorem \ref{thm_carl} with $p,q$ satisfying \eqref{critical} gives \eqref{carl_el} for $p,q$ satisfying $\ppq\in (\mathscr B, \mathscr B')$.
However, it does not extend the previously known range of $p,q$ for which  \eqref{carl_el}  holds.   When  $d\ge 5,$ 
the range of $p,q$ coincides with that in Kwon and Lee \cite{KL18}, which was obtained by making use of  
the sharp estimate for the spherical harmonic projection.  The optimal range  for the estimate \eqref{carl_el} still remains open.

To obtain \sucp for potentials in  $L_{t,loc}^sL_x^{r,\infty}$ we need to obtain \eqref{carl} with $a=b$. To this end, we  are basically relying on real interpolation to upgrade 
$L^r_tL_x^{p}$--$L^s_tL_x^{q}$ estimates to  these of $L^r_tL_x^{p,a}$--$L^s_tL_x^{q,b}$ with $a=b$.  However,  such extension of the Carleman inequality \eqref{carl}  to  the Lorentz spaces is not so straightforward as in \cite{St85} since  real interpolation  does not behave well in mixed norm spaces (see \cite{cw}).  In particular, 
we are only able to  obtain  \eqref{carl}  with $a=b=2$ when $p,q$ satisify \eqref{critical} (also see Lemma \ref{carl_pf}).

\subsubsection*{Strong unique continuation property for the heat operator} 
The extension of the Carleman estimate to the Lorentz spaces  (Theorem \ref{thm_carl}) allows 
a larger class of potentials for the  strong unique continuation property  for the heat operator.  
In this regard we obtain Theorem \ref{co_suc} and Theorem \ref{co_suc1} below  which  improve  the results in  \cite{ev01}.  
Once we have the Carleman estimate \eqref{carl}, Those theorems  can be shown by routine adaptation of the argument  in \cite{ev01}. 
So, we state them without providing  proofs.  

\begin{thm}\label{co_suc} Let $d\ge3$, $0<T<\infty,$  and 
$\mathfrak r,\mathfrak s$ satisfy \eqref{rs-con}. Let $\ppq\in \mathfrak T$  satisfy $1/p-1/q=1/\mathfrak r$.  Suppose that  $u\in W^{1,a}((0,T); W^{2,p}(\R^d))$, $a\le \min\{2,\mathfrak s\}$, is a  solution to the differential inequality \eqref{diff_ineq}
and suppose that for any $k\in \mathbb N$  there is a constant $C_k$ such that \eqref{parabolic-decay} holds for some $\delta>0$. 
Then $u$ is identically zero on $\R^d\times (0,T)$ provided that  $\| t^{1-\frac d{2\mathfrak r}-\frac1{\mathfrak s}}V \|_{L^{\mathfrak s}((0,T); L_x^{\mathfrak r,\infty}(\R^d))}$ is small enough. 
\end{thm}

Most significantly, Theorem  \ref{co_suc} gives  the \sucp 
with $V\in L^{\infty}((0,T); L_x^{d/2,\infty}(\R^d))$. This strengthens the result obtained by 
Escauriaza and Vega \cite{ev01}  under the assumption that    $\|V \|_{L^{\infty}((0,T); L_x^{d/2})}$ is small enough.  Using Wolff's construction in  \cite{W92} we can show that the smallness assumption is necessary in general for $V\in L^{\infty}((0,T); L^{d/2,\infty}(\R^d))$, or $V\in  L^{d/2,\infty}(\R^d; L^{\infty}((0,T)))$.  Indeed,  Wolff showed that there is  a bounded nonzero function $w$  such that  $|\Delta w|\le |V_\ast w|$ with $V_\ast\in L^{d/2,\infty}$ and  vanishes to infinite order at the origin.  Since the function $w$ in \cite{W92}  is bounded, considering the time independent function $u(x,t):=w(x)$  it is easy to see that $u(x,t)$ satisfies \eqref{parabolic-decay} and  obviously the differential inequality $|\Delta u+\partial_t u|\le |V_\ast u|$.   

We also have the following  \sucp result for a local solution.

\begin{thm}\label{co_suc1} Let $d\ge 3$ and $\mathfrak r,\mathfrak s$ satisfy \eqref{rs-con}. 
Suppose that $u$ is a continuous solution to $|\Delta u+\partial_t u|\le |Vu|$ on $B(0,2)\times (0,2)$ and suppose that  
 for any $ k\in \mathbb N$ there is a constant $C_k$ such that
\[ \|e^{-|x|^2/8t} u\|_{L^2((0,\varepsilon );L_x^2(B(0,2)))} \le C_k \varepsilon^k,\quad 0<\varepsilon<2 . \]
Then $u(x,0)$ vanishes on $B(0,2)$ if $\|t^{1-\frac d{2 \mathfrak r}-\frac1{\mathfrak s}} V\|_{L^{\mathfrak s}((0,2); L_x^{\mathfrak r,\infty}(B(0,2)))}$   is small enough. 
\end{thm}

\subsubsection*{Uniform resolvent estimate for the Hermite operator}  
We now consider the  resolvent estimate for the Hermite operator  $H=-\Delta+|x|^2$  in $\mathbb R^d$: 
\begin{equation}
\label{eq_resol}
\|(H-z)^{-1} f \|_{q}\le C\|f\|_p,\quad z\in \mathbb C\setminus (2\mathbb N_0+ d)
\end{equation}
with a constant $C$ independent of $z$.  The estimate  has  independent interest while it plays an important role in proving Theorem \ref{thm_carl} (see Lemma \ref{carl_pf}). 
Since  $H$ has the discrete  spectrum $2 \N_0 +d$, $z\in  2\mathbb N_0+ d$ are excluded. 
 In contrast with the operator with a continuous spectrum, it is impossible for \eqref{eq_resol} to hold with $C$ independent of $z$, so we need to impose 
the assumption that 
\Be\label{eq:z-dist} \dist( z, 2\mathbb N_0+d )\ge c\Ee
 for some $1\gg c>0$.  (See Remark \ref{re_gap}). The estimate \eqref{eq_resol} may be compared with the corresponding estimate for the resolvent of  the Laplacian  which is due to Kenig, Ruiz, and Sogge \cite{KRS87}. It was shown in \cite{KRS87} that the estimate 
\[
	\|(-\Delta-z)^{-1}f\|_{q}\le C\|f\|_p,  \quad  z\in \mathbb C\setminus (0,\infty)
\]
holds with $C$ independent of $z$ if and only if $1/p-1/q=2/d$, ${2d}/(d+3)<p <{2d}/(d+1)$ and $d\ge 3$. 
Also, see \cite{JKL16} for the uniform estimates for more general second order differential operators  and \cite{KL19}  for the  sharp bounds which depend on $z$.  
 Under the assumption  \eqref{eq:z-dist},  the uniform resolvent estimate for $H$ continues to hold with $p,q$ away from the critical line $1/p-1/q=2/d$ whereas  this can not be true for $-\Delta$ because of scaling structure  (see \cite{KRS87, KL19}).

The uniform estimate \eqref{eq_resol} was obtained  by 
 Escauriaza and Vega \cite{ev01}  for ${2d}/{(d+2)}\le p\le 2\le q\le {2d}/{(d-2)}$, $d\ge3$.  
However,  $\eqref{eq_resol} $ fails to hold if $1/p-1/q>2/d$ (see Remark \ref{failure}) and  the proof of  \eqref{eq_resol}  is  more involved if  $p,q$ satisfy  \eqref{critical}.
 As for such $(p,q)$ of the critical case  the estimate has been  known only for $(p,q)=({2d}/{(d+2)}, {2d}/{(d-2)})$.   In what follows we establish  \eqref{eq_resol}  for $\ppq \in (\mathscr B, \mathscr B ')$. Those estimates in the expanded range  are crucial for obtaining \eqref{carl}  with
 $a=b$ when $p,q$ satisfy  \eqref{critical}.  

  \begin{thm}\label{thm_resol}
Let $d\ge 3$. Suppose $\ppq\in \mathfrak T$ and  \eqref{eq:z-dist} holds.  
Then, there is a constant $C>0$ such that 
\eqref{eq_resol} holds. Furthermore, if $\ppq=\mathscr B $ or $\mathscr B '$, we have  restricted weak type $($uniform$)$ estimate for $(H-z)^{-1}$.  
\end{thm}

The proof of the  estimate \eqref{eq_resol} with $(p,q)=({2d}/{(d+2)}, {2d}/{(d-2)})$ 
in Escauriaza and Vega  \cite{ev01}  heavily relies on  the 
uniform bound on the spectral projection operator $\Pi_k$ which is the projection onto the $k$-th eigenspace of the Hermite operator $H$ (see Section 2). 
In fact,  they also used  interpolation along an analytic family of operators which are motivated by Mehler's formula for the Hermite function. However, 
their argument is not enough to prove  \eqref{eq_resol}  for $\ppq\in(\mathscr B,\mathscr B')$. We develop a different approach which is more direct and significantly simpler. 
We make use of a representation formula \eqref{eq_proj} for $\Pi_k$ which was observed  in \cite{JLR21} and an estimate  for the Hermite-Schr\"odinger propagator $e^{-itH}f$ (see Proposition \ref{st-est})
which is a consequence of the representation formula and the endpoint Strichartz estimate \cite{keel-tao}.

{\it Organization of the paper.} The rest of this paper is organized as follows. In Section 2 we provide useful properties of the Hermite operator $H$ and the Hermite spectral projection operator $\Pi_k$.
We prove boundedness of more general multiplier operator for the Hermite operator in Section 3, which implies Theorem \ref{thm_resol}. Finally, the proof of the Carleman estimate for the heat operator is given  in Section 4.

\section{Properties of the Hermite  operator}

 For any multi-index $\alpha\in \N_0^d$ the $L^2$-normalized Hermite function $\Phi_\alpha$ which is a tensor product of one dimensional Hermite functions is an eigenfunction of $H$ with eigenvalue $2|\alpha|+d$. Here $|\alpha|:=\alpha_1+\cdots +\alpha_d$. The 
set  $\{\Phi_\alpha: \alpha\in\N_0^d\}$ forms an orthonormal basis of $L^2(\R^d)$. Thus, for any $f\in L^2(\R^d)$ we have the Hermite expansion     
$f= \sum_\alpha \langle f, \Phi_\alpha\rangle \Phi_\alpha$. 

We consider the Hermite spectral projection operator $\Pi_k$ which is  defined by
\[\Pi_k f= \sum_{\alpha\in\N^d_0 : |\alpha|=k} \langle f, \Phi_\alpha\rangle \,\Phi_\alpha,\quad f\in \mathcal S(\R^d).\]
Then, the Hermite-Schr\"odinger propagator is given by 
\[ e^{-itH}f=\sum_{k\in \N_0} e^{-it(2k+d)}\Pi_k f,\quad f\in \mathcal S(\R^d),\]
which is the solution to the Cauchy problem $(i\partial_t- H)u=0$, $u(x,0)=f(x)$. 
If $f\in \mathcal S(\R^d)$, it is easy to see that $\Pi_k f$ decays rapidly in $k$, thus  $\sum_{k=0}^\infty e^{-it(2k+d)}\Pi_k f$  converges uniformly. 
Clearly, $\Pi_k f = \sum_{k'\in \N_0}\frac1{2\pi} (\int_{-\pi}^\pi e^{it(k-k') } dt) \Pi_{k'}f$. Therefore, we obtain
\begin{equation}\label{eq_proj}
\Pi_k f =  \frac1{2\pi} \int_{-\pi}^\pi e^{i\frac t2(2k+d -H) }f dt
\end{equation}
for $f\in \mathcal S(\R^d).$ Meanwhile, the operator $e^{-itH}$  has the kernel formula
\begin{equation}\label{eq_schr}  e^{-itH}f(x) =C_d (\sin 2t)^{-\frac d2} \int_{\R^d} e^{i(\frac{|x|^2+|y|^2}{2}\, \cot 2t -\inp x y\,\csc2t)} f(y)dy
\end{equation}
for $f\in \mathcal S(\R^d)$, which is shown by making use of Mehler's formula (\cite{SoTo,Th87}).  Combining this with \eqref{eq_proj} gives an explicit expression of the kernel of $\Pi_k.$

In order to prove the uniform resolvent estimate (Theorem \ref{thm_resol}) we make use of the following mixed norm estimate for $e^{-it H}$, which strengthens the uniform bound \eqref{unif} in a different direction. 
\begin{prop}\label{st-est} Let $d\ge 3$ and  $\ppq=\mathscr B'$. Then, we have 
\begin{equation}\label{eq_mixed}
\| \int_{-\pi}^{\pi} |e^{-i\frac t2H}f|dt\|_{q,\infty} \le C \|f\|_{p,1}. 
\end{equation}
\end{prop}

Various authors  (see \cite{K94, Th98,  T05, JLR21}) studied the problem of characterizing the sharp asymptotic bound on the operator norm  $\|\Pi_k\|_{p\to q}$  of $\Pi_k$ from $L^p$ to $L^q$ as $k\to \infty$.  In particular, 
Karadzhov \cite{K94} showed 
\begin{equation}\label{unif}
 \|\Pi_k\|_{p\to q}\le C
 \end{equation}
for a constant $C$ when $p=2$ and $q={2d}/{(d-2)}$. 
By duality and $TT^*$-argument,  the bound \eqref{unif} with $(p,q)=({2d}/{(d+2)},2)$ and $(p,q)=({2d}/{(d+2)},{2d}/{(d-2)})$ follows. Interpolating those estimates with 
the trivial bound $\|\Pi_k\|_{2\to 2}\le 1$, we have \eqref{unif}  for $p,q$ satisfying ${2d}/{(d+2)}\le p\le 2\le q\le {2d}/{(d-2)}.$

Recently,  the authors \cite[Theorem1.2]{JLR21}  showed that  \eqref{unif} holds on an extended range of $p,q$ for $d\ge3$ (see \cite{JLR20} for a related result). 
By means of Proposition \ref{st-est} we can provide a simple alternative proof of this result. Indeed, from  \eqref{eq_proj} and Proposition \ref{st-est}  it follows that  $\|\Pi_kf\|_{q,\infty}\le C\|f\|_{p,1}$ if $(1/p,1/q)=\mathscr B'.$  
By duality, the same estimate also holds for $(1/p,1/q)=\mathscr B.$  Interpolating these estimates with the above mentioned estimate \eqref{unif} for ${2d}/{(d+2)}\le p\le 2\le q\le {2d}/{(d-2)}$  gives the following.
(See Figure \ref{fig:uniform}.)

\begin{cor}$($\cite[Theorem1.2]{JLR21}$)$ 
\label{glo-est} Let $d\ge 3.$ For $p,q$ satisfying $\ppq\in \mathfrak T$ there is a constant $C>0$, independent of $k$, such that 
\eqref{unif} holds.
 Furthermore, the uniform  restricted weak type estimate for $\Pi_k$ holds if $(\frac1p,\frac1q)=\mathscr B$ or $\mathscr B'$. 
\end{cor}

\begin{proof}[Proof of Proposition \ref{st-est}] 
We make use of the endpoint Strichartz estimate for $e^{-itH}$:
\begin{equation}\label{eq_endpt}
\| e^{-i\frac t2 H} f\|_{L^2_t([-\pi, \pi]; L^{p_\circ}_x(\R^d))} \le C\|f\|_2
\end{equation}
at $p_\circ =\frac{2d}{d-2}$, which is can be shown by the dispersive estimate from \eqref{eq_schr} and the standard argument in \cite{keel-tao} (for example, see \cite{SoTo}). 
We choose a smooth partition of unity so that 
\[
 \psi^0+ \sum_{j\ge 4} \big( \psi(2^j t)+ \psi(-2^j t)  +\psi(2^j(t+\pi))+\psi(2^j(\pi-t))\big)=1
 \]
for $ t\in(-\pi,\pi)\setminus\{0\}$.  Here $\psi\in \mathrm C_c^\infty([\frac14,1])$ satisfying  $\sum_j \psi(2^j t) =1$ for $t>0$, and $\psi^0 $ is a smooth function which is supported in the interval $[-\pi, \pi]$ and vanishes near  $0, \pi,$  and $-\pi$.  

Set $\psi_j^\pm=\psi(\pm 2^j\cdot)$ and  $\psi_j^{\pm\pi}=\psi(2^j(\pi-\pm\, \cdot))$. Then, for $\sigma=\pm, \pm\pi$, we have $\int |\psi_j^\sigma e^{-i\frac t2H}f| dt \lesssim 2^{\frac {d-2}2 j} \|f\|_1$ because $|\psi_j^\sigma e^{-i\frac t2H}f|\lesssim 2^{\frac d2 j}\|f\|_1$ by \eqref{eq_schr}. 
By using \eqref{eq_endpt} and H\"older's inequality followed by Minkowski's inequality 
we also get $\|\int |\psi_j^\sigma e^{-i\frac t2H}f| dt\|_{\frac{2d}{d-2}} \lesssim 2^{-\frac {1}2 j}\|f\|_2$.    Interpolation among those estimates gives 
\[ \| \int |\psi_j^\sigma e^{-i\frac t2H}f| dt \|_q\lesssim 2^{(\frac {d}2(\frac1p-\frac1q)-1) j} \|f\|_p, \quad  \sigma=\pm, \pm\pi \]
 if $\ppq$ is contained in the line segment $[(1,0), (1/2, (d-2)/2)]$.  
Bourgain's  summation trick (for example see  \cite[Lemma 2.4]{JLR21})  to the above estimates gives 
\[ \| \int |\sum_j\psi_j^\sigma e^{-i\frac t2H}f| dt \|_{q,\infty}\lesssim  \|f\|_{p,1}, , \quad  \sigma=\pm, \pm\pi\]
for $(1/p,1/q)=\mathscr B'.$  By a similar argument,  it is easy to show $\| \int |\psi^0 e^{-i\frac t2H}f| dt \|_{q}\lesssim  \|f\|_{p}$ for $(1/p,1/q)=\mathscr B'.$  Hence, combining all of those estimates, we get \eqref{eq_mixed}.
\end{proof}

We now consider $L^p$--$L^q$ estimate for the  operator $H^{-s}$, $s>0$ which is defined by $H^{-s} f =\sum_{k=0}^\infty (2k+d)^{-s} \Pi_kf.$
The operator  can also be written as
\[ H^{-s}f =\frac1{\Gamma(s)}\int_0^\infty t^{s-1} e^{-tH}f\ dt \,\footnote{For a bounded function $\mathfrak m$ on $\mathbb R_+$ the operator $\mathfrak m(H)$ is formally defined by  $\mathfrak m(H)= \sum_{k\in\N_0} \mathfrak m(2k+d)\Pi_k$.}
 \]making use of the heat semigroup $e^{-tH}$ associated to $H$.  By means of the explicit kernel expression of $e^{-tH}$ which is  based on Mehler's formula (see \cite{Th93}),  Bongioanni and Torrea \cite{BoTo} obtained  $L^p$--$L^q$ boundedness for $H^{-s}$.  Sharpness of their result was later verified by  Nowak and  Stempak \cite{NS}. Thus, the results completely characterize  $L^p$--$L^q$ boundedness of  $H^{-s}$. 

\begin{thm} $($\cite[Theorem 8]{BoTo},\,\cite[Theorem 3.1]{NS}$)$\label{sob}
Let $d\ge 1$, $1< p, q<\infty,$ and\, $0<s<d/2.$ Then,   $H^{-s}$ is bounded from $L^p$ to $L^q$ if and only if  $-{2s}/d<1/p-1/q\le  {2s}/d$.
\end{thm}

There are weak/restricted weak type  estimates in the borderline cases which are not included in the above theorem, and we refer the readers to \cite{NS} for more details regarding such endpoint estimates.

\section{Proof of Theorem \ref{thm_resol}}

We consider more general operator $(H-z)^{-m}$, $m\in \mathbb N$ which is  given by
\[
(H-z)^{-m} f =\sum_{k=0}^\infty \frac{\Pi_k\,f}{(2k+d-z)^{m}}=(-2)^{-m}\sum_{k=0}^\infty \frac{\Pi_k\,f}{(i\tau +\beta-k)^{m}}
\]
with 
$z=2\beta+d+2\tau i$ and $\beta\not\in\mathbb N_\circ$. We prove the following. 

\begin{thm}\label{uniform} 
Let $d\ge 3$ and let $m$ be a positive integer. Suppose that  \eqref{eq:z-dist} holds for some $c>0$. If $\ppq\in(\mathscr B,\mathscr B')$, then there is a constant $C=C(m)$, independent of $z$, such that  
\Be
\label{h-resolvent}
\|(H-z)^{-m} f\|_q\le C (1+|\im z|)^{1-m}\|f\|_p 
\Ee
 Furthermore, if $\pq=\mathscr B$ or $ \mathscr B'$, then we have   the restricted weak type estimate  $\|(H-z)^{-m} f\|_{q,\infty}\le C(1+|\im z|)^{1-m}\|f\|_{p,1} $. 
 \end{thm}
 
 While the estimates for  $m\ge 2$ are rather straightforward from the \eqref{unif}, the proof of \eqref{h-resolvent} for  $m=1$  is more involved.  
 This  case  is  handled  in Proposition \ref{ggg} below.

 \begin{rem}\label{re_gap} The gap condition \eqref{eq:z-dist} is necessary for the uniform estimate \eqref{h-resolvent} to hold.  In fact,  $ \|(H-z)^{-m}\|_{p \to q} \ge |2k+d-z|^{-m}\|f\|_q/\|f\|_p$  if $f$ is an eigenfunction with eigenvalue $2k+d.$   
Therefore,  the operator norm can not be bounded as  $z\to 2k+d$  unless \eqref{eq:z-dist} holds. 
 \end{rem}

For $\mathcal B, t_\circ>0$,   let $\mathcal C(\mathcal B,t_\circ)$  denote the class of functions on $\mathbb R$ which satisfy
  \begin{align}
 & \ |G(n)|\le \mathcal B,  \quad  n\in \mathbb Z;
 \label{g-con1}
 \\
 &\sum_{k=1}^\infty |G(k)+G(-k)|\le \mathcal B;
 \label{g-con2}
 \\
&\sum_{k=1}^\infty |kG(k)-(k+1)G(k+1)|\le \mathcal B;
\label{g-con22}
\\
& \Big |\Big( \frac{d}{dt}\Big)^l G( t)\Big  |\le \mathcal B  (1+  |t|)^{-l-1}, \quad t_\circ< |t|
 \label{g-con3} 
 \end{align}
 for $0\le  l\le (d+2)/2$. 
 Particular examples satisfying the conditions \eqref{g-con1}--\eqref{g-con3} are  $G_{\mu, \tau}(t)=1/(i\tau +t +\mu)$ where  $(\mu,\tau)\in( -\frac12,\frac12)\times \R$ and $|(\mu,\tau)|\ge c$ for some small $c>0$.

\begin{prop} 
\label{ggg} 
Let $d\ge 3$  and $\ppq\in(\mathscr B ,\mathscr B')$.  
Suppose that $G$ is in $\mathcal C(\mathcal B,t_\circ)$. Then, for any $n\in \mathbb N_0$, there is a constant $C$, depending only on  $\mathcal B$ and $t_\circ$, such that 
\Be\label{eq_G}
\Big\| G\Big(\frac{2n+ d-H}2\Big) f\Big\|_q\le C  \|f\|_p.
\Ee
Furthermore, if $\ppq=\mathscr B $ or $ \mathscr B' $,   the restricted weak type $(p,q)$ estimate holds for $G(\frac{2n+ d-H}2)$ with a uniform bound. \end{prop}

\begin{proof}
Let $p_\ast$ and $q_\ast$ be given by $(1/p_\ast, 1/q_\ast)=\mathscr B'$.  In order to show Proposition \ref{ggg}  it is sufficient  to show the restricted weak type $(p_\ast,q_\ast)$ estimate   for  $G(\frac{2n+ d-H}2)$. Note that the adjoint operator $G(\frac{2n+ d-H}2)^\ast$ is given by $G(\frac{2n+ d-H}2)^\ast f=\sum_{k=0}^\infty \ol G(n-k) \Pi_kf$.  Then, clearly $\ol  G\in \mathcal C(\mathcal B,t_\circ)$. Hence,  the same argument shows that  restricted weak type $(p_\ast,q_\ast)$ estimate  holds  for  $G(\frac{2n+ d-H}2)^*$. This in turn gives  the restricted weak type estimate $(q_\ast', p_\ast')$  for  $G(\frac{2n+ d-H}2)$ by duality. Real interpolation between these two (restricted weak type) estimates for $G(\frac{2n+ d-H}2)$  yields the desired estimates for $\ppq\in(\mathscr B ,\mathscr B')$.

No  differentiability  assumption is made on $G$ for  $|t|\le t_\circ$. So, we handle the cases $n\ge n_\circ$ and $n<n_\circ$ separately, where  $n_\circ$  is an integer satisfying $n_\circ\ge2t_\circ.$ 
We first consider the case $n\ge n_\circ$.  Recalling  $G(\frac{2n+ d-H}2)=\sum_{k=0}^\infty G(n-k) \Pi_k$,   we decompose 
 \begin{align*}
 G\Big(\frac{2n+ d-H}2\Big)=:\mathcal J_n+\mathcal K_n, 
 \end{align*}
 where
\begin{align*}
\mathcal J_{n}&:=\sum_{k=0}^\infty  G(n-k) \phi\Big(\frac{n-k}n\Big)\Pi_k ,
\\
\mathcal K_n&:=\sum_{k=0}^\infty G(n-k) \Big(1-\phi\Big(\frac{n-k}n\Big)\Big)\Pi_k. 
\end{align*}
Here,  we choose  a non-negative smooth even function  $\phi$ on $\mathbb R$ such that $\phi(t)=1$ on $[-1/2,1/2]$,  $\phi=0$ if $1\le |t|$, and $\phi$ is non-increasing on the half-line $t>0$. This monotonicity assumption plays an  important role in estimating a sum of trigonometric functions. 

The $\mathcal J_n$  is the major contribution to the estimate \eqref{eq_G} and is to be handled by  the integral formula for $\Pi_k$  and Lemma \ref{st-est}. The second   $\mathcal K_n$ behaves like the operator $H^{-1}$, which is actually bounded from $L^p$--$L^q$ on a larger range of $p,q$. We consider $\mathcal J_n$ first. 

We  set 
\begin{align*}
 \CI_1&= \sum_{k=1}^n  G(k) \phi({k}/ n)(\Pi_{n-k}-\Pi_{n+k}),
\\
 \CI_2&=\sum_{k=1}^n  (G(-k)+  G(k)) \phi({k}/ n)\Pi_{n+k}.
 \end{align*}
Since $\phi$ is even function and supported in $[-1,1]$, after reindexing by $(n-k)\to k$   we see 
$\mathcal J_n =
\sum_{k=1}^n  G(k) \phi({k}/ n)\Pi_{n-k}+ G(0) \Pi_n+ \sum_{k=1}^n  G(-k) \phi({k}/ n)\Pi_{n+k}$. Thus,  
\begin{align*}
\mathcal J_n =\CI_1+\CI_2+ G(0) \Pi_n
.
\end{align*}
By   \eqref{g-con1},   \eqref{g-con2}, and  the uniform restricted weak type $(p_\ast, q_\ast)$ estimate for $\Pi_\lambda$ in Corollary \ref{glo-est},
it follows that  $\|G(0) \Pi_n f\|_{\qs,\infty}\lesssim   \mathcal B\|f\|_{\ps,1} $ and $\|\CI_2 \|_{\qs,\infty}\lesssim   \mathcal B\|f\|_{\ps,1} $. 
So, it suffices to deal with the first term $\CI_1$.  Using the formula \eqref{eq_proj},  we note  $\Pi_{n-k}f-\Pi_{n+k}f=-\frac i\pi \int_{-\pi}^\pi  \sin (tk) e^{i\frac t2(2n+d -H)}fdt$. Thus,   we have 
\begin{align*} 
\CI_1 f =\int_{-\pi}^\pi \zeta_n(t)e^{-i\frac t2H}f dt,
\end{align*}
where 
\[\zeta_n(t) =-\frac i\pi e^{i\frac t2(2n+d)}\sum_{k=1}^n G(k)   \sin (tk) \phi({k}/ n) ,\quad -\pi \le t\le \pi.\]
Using Proposition \ref{st-est},  it is sufficient to show  \begin{equation}\label{alter}
 |\zeta_n(t) |\le C
 \end{equation}
with $C$ independent of  $n$ and $G$. By the property of $\phi$ which we have chosen,  it is clear  that 
$
|\zeta_n(t) |\lesssim  |\sum_{k=1}^{\lfloor \frac n2\rfloor}   \sin (tk)G(k)|
+|\sum_{k=\lfloor \frac n2\rfloor+1}^n \sin (tk)G(k)\phi({k}/ n)|. 
$ 
Boundedness of the second term is easy to show. Indeed, since the condition \eqref{g-con3} holds for $|t|> {n}/2$ by our choice of $n_\circ$, 
we see 
\[|\sum_{k=\lfloor \frac n2\rfloor+1}^n \sin (tk)G(k)\phi({k}/ n)|\lesssim \mathcal B\sum_{k=\lfloor \frac n2\rfloor+1}^n {k^{-1}}\phi( k/n)\lesssim \mathcal B.\]
So, for \eqref{alter} we only have to show  the estimate 
$
| \sum_{k=1}^n \sin (tk)G(k)|
\lesssim 1
 $
for any $n$.  Setting $\sigma_k(t)=\sum_{j=1}^k j^{-1}{\sin (jt)}$,  
 by summation by parts we write 
\[\sum_{k=1}^n \sin (tk)G(k)=\sum_{k=1}^{n-1}\sigma_k(t)\Big(k G(k) -(k+1)G(k+1)\Big)+  \sigma_n(t) nG(n).\] 
Since $|\sigma_k(t)|\lesssim 1$ for any $k, t$ as can be shown  by an elementary argument,\footnote{This can be seen by approximating Dirichlet's kernel, or again by summation by parts.} by 
the conditions \eqref{g-con22} and \eqref{g-con3} it follows that $
| \sum_{k=1}^n \sin (tk)G(k)|
\lesssim 1
 $.

We now turn to the operator $\mathcal K_n$. Clearly, we may write
$\mathcal K_n =H^{-1} \circ m_n(H)$
where  $m_n$ is given by 
\[m_n(t) =  t \,G\big({(2n+d-t)}/2\big) \big(1-\phi ((2n+d-t)/{2n})\big),\] 
which is in $C^\infty(\mathbb R).$
Using \eqref{g-con1}, \eqref{g-con3}, and the support property of $\phi$,  a  simple calculation shows $|\frac {d^l}{dt^l} m_n(t)|\lesssim (1+t)^{-l}$ for 
 $l=0,1,2,\cdots,(d+2)/2$ whenever $t>0$ and  the  implicit constants are independent of $n$. Thus, the Marcinkiewicz multiplier theorem \cite[Theorem 4.2.1]{Th93} implies that $m_n(H)$ is bounded on $L^p$, $1<p<\infty,$ uniformly in  $n$.  By Theorem \ref{sob}, $H^{-1} $ is also bounded from $L^p$ to $L^q$ for $1<p,q<\infty$ satisfying $1/p-1/q=2/d$.  Hence, we have
\[ \|\mathcal K_n\|_{p\to q}\le \|H^{-1}\|_{p\to q}  \|m_n(H)\|_{p\to p}\lesssim 1\]
with the implicit constant independent of $n$.

We now consider the case $n<n_\circ$, which  is much simpler to show than the case $n\ge n_\circ$. To prove \eqref{eq_G}, we break $G(\frac{2n+ d-H}2)$ as follows:
\begin{align*}
 G\Big(\frac{2n+ d-H}2\Big) =\widetilde {\mathcal J}_n+\widetilde {\mathcal K}_n, 
 \end{align*}
where 
\begin{align*}
\widetilde {\mathcal J}_{n}&=\sum_{k=0}^\infty  G(n-k) \phi (k/{2n_\circ})\Pi_k ,
\\
\widetilde  {\mathcal K}_n&=\sum_{k=0}^\infty G(n-k) \big(1-\phi(k/{2n_\circ})\big)\Pi_k. 
\end{align*}
Clearly,  the multiplier $G\big({(2n+d-\cdot)}/2\big) \big(1-\phi ((2n+d-\cdot)/{2n_\circ})\big)$ of the operator $\widetilde{\mathcal K}_n$ satisfies the condition \eqref{g-con3}. So, 
in the same manner as in the above we obtain the bound $\|\widetilde {\mathcal K}_n\|_{p\to q}\lesssim 1$  if $1<p,q<\infty$ and $1/p-1/q=2/d$. By the condition \eqref{g-con1} and Corollary \ref{glo-est} 
it follows that $\|\widetilde{\mathcal J}_n f\|_{\qs,\infty} \le \mathcal B \sum_{k=0}^{2n_\circ}\|\Pi_kf\|_{\qs,\infty}\lesssim \|f\|_{\ps,1}$ 
uniformly in $n\le n_\circ$. This  completes the proof of Proposition \ref{ggg}. 
\end{proof}

We are ready to prove Theorem \ref{uniform}.

\begin{proof}[Proof of Theorem \ref{uniform}]  
Let $p_\ast$ and $q_\ast$ be given by $(1/p_\ast, 1/q_\ast)=\mathscr B'$. As in the proof of Proposition \ref{ggg},   it is enough to show 
the restricted weak type $(p_\ast, q_\ast)$ estimate for $(H-z)^{-m} $ with bound $C (1+|\im z|)^{1-m}$ since the adjoint operator of  $(H-z)^{-m} $ is given by $(H-{\ol z})^{-m}$.  We can handle  $(H-\ol z)^{-m} $ in the exactly same way to obtain the restricted weak type $(p_\ast, q_\ast)$ estimate with bound $C (1+|\im z|)^{1-m}$. By duality and interpolation, we get all the desired estimates. 

By Corollary \ref{glo-est}, we have the estimate $\| \Pi_k f\|_{\qs,\infty}\le C\|f\|_{\ps,1}$ with $C$ independent of  $k$. Using this estimate, for $m\ge 2$ we get
\[ \|(H-z)^{-m} f\|_{\qs,\infty} \lesssim \sum_{k=0}^\infty |2k+d-z|^{-m} \|f\|_{\ps,1}\lesssim  (1+|\im z|)^{1-m}  \|f\|_{\ps,1}\] 
because $\sum_{k=0}^\infty |2k+d-z|^{-m} \le C_m (1+|\im z|)^{1-m}$   with $C_m$ independent of $z$ for  $m\ge2$ if \eqref{eq:z-dist} holds. Thus we need only to show
\Be \label{eq:rw} \|(H-z)^{-1} f\|_{\qs,\infty} \le C  \|f\|_{\ps,1}.\Ee
 If $ \re z>d-1$,   $z=2(n+\mu)+d +2i\tau$ for some $n\in \mathbb N_0$, $\mu\in(-\frac12,\frac12)$, and $\tau\in\R$ satisfying $|(\mu, \tau)|\ge c/2$ because of \eqref{eq:z-dist}.
 We note that 
 \[ (H-z)^{-1} = G_{\mu,\tau}\Big(\frac{2n+ d-H}2\Big)\] 
where  $G_{\mu, \tau}(t)=1/(i\tau +t +\mu)$. It is easy to see that  $G_{\mu, \tau}  \in \mathcal C(\mathcal B, 1)$ for some $\mathcal B>0$ provided that 
$\mu\in(-\frac12,\frac12)$, and $\tau\in\R$ satisfy $|(\mu, \tau)|\ge c/2$. Thus, by Proposition  \ref{ggg} the estimate \eqref{eq:rw}  holds uniformly in $z$. For the remaining case, i.e., $ \re z<d-1$, $z$ clearly stays  away from the eigenvalues of $H$, so $(H-z)^{-1}$ behaves like $H^{-1}$.  More precisely, we obtain the uniform estimate \eqref{eq:rw}  repeating the same argument as in the case $n<n_\circ$ of the proof of Proposition \ref{ggg}. This completes the proof. 
\end{proof}

The uniform resolvent estimate  in Theorem \ref{thm_resol} is a special case of the following.

\begin{cor}\label{co_uniform}
Let $d\ge 3$ and $m$ be a positive integer, and  let  $p,q$  be given  as in Theorem \ref{thm_carl}. 
Then, there is a constant $C=C(m)$ such that 
\Be
\label{h-resolventg}
\|(H-z)^{-m} f\|_q\le C (1+|\im z|)^{\frac d2(\frac1p-\frac1q)-m}\|f\|_p
\Ee
provided \eqref{eq:z-dist} holds. Furthermore, if $(\frac1p,\frac1q)=\mathscr B $ or $\mathscr B' $, we have the restricted weak type estimate for $(H-z)^{-m}$  
with  bound $C (1+|\im z|)^{\frac d2(\frac1p-\frac1q)-m}$.  
\end{cor}

\begin{proof}
By Theorem \ref{uniform} we have the estimate \eqref{h-resolventg}  for $\ppq\in (\mathscr B, \mathscr B')$. 
In view of  interpolation, it is enough to show  \eqref{h-resolventg} with $(p,q)=(2,2)$, $(\frac{2d}{d+2},2)$ or $(2,\frac{2d}{d-2})$. 
These estimates are easy to show using  orthogonality between the projection operators  $\Pi_k$. In fact,  we have
\[ \|(H-z)^{-m} f\|_2 \le \Big(\sum_{k=0}^\infty |2k+d -z|^{-2m} \|\Pi_k\, f\|_2^2\Big)^{1/2}.
 \]
So, taking the supremum over $k$ of $|2k+d -z|^{-2m}$, we obtain \eqref{h-resolventg} when $p=q=2$.
We note that $ \sum_{k=0}^\infty |2k+d -z|^{-2m}\le C(1+|\im z|)^{-2m+1}$ with $C$ independent of  $z$ as long as \eqref{eq:z-dist} holds.  Applying the uniform
$L^\frac{2d}{d+2}$--$L^2$ estimate in  Corollary \ref{glo-est},  we get \eqref{h-resolventg} with $p=\frac{2d}{d+2}$ and $q=2.$ Since the adjoint of $(H-z)^{-m}$ is $(H-\bar z)^{-m},$ 
the estimate \eqref{h-resolventg} with $(p,q)=(\frac{2d}{d+2},2)$ implies that with  $(p,q)=(2,\frac{2d}{d-2})$ by duality.
 \end{proof}

\section{Proof of Theorem \ref{thm_carl}}  
 We now prove the estimate \eqref{carl} by adapting the argument in   Escauriaza and Vega \cite{ev01} (also see \cite{e00}) which  deduces Carleman estimate  for the heat operator  from the uniform resolvent  estimate for the Hermite operator.  We are basically relying on real interpolation as in \cite{St85}. However,  there are some nontrivial issues which are related to the shortcoming  of the real interpolation between mixed norm spaces. 

\begin{lem}\label{carl_pf} 
Let $1< p\le 2\le q< \infty$,  $1\le r,s\le \infty$, $1\le a\le b\le \infty$, and let $0\le \gamma\le 1$ and $\beta\notin \mathbb N_0$ be a  real number. 
Suppose that  the estimate 
\begin{equation}\label{eq_uni}
\Big\|\sum_{k=0}^\infty \frac{\Pi_k\,f}{(\tau i+\beta -k)^{m}}  \Big\|_{{q,b }}\le C_m(1+|\tau|)^{\gamma-m} \|f\|_{p,a} 
\end{equation}
holds for  $m=1,2,3$ with $C_m$ independent of  $\tau \in\R$ and $\beta$ provided $\dist(\beta,\mathbb N_0)\ge c$ for some $c>0.$  Then, if $\dist(\beta,\mathbb N_0)\ge c$ for some $c>0$,   the estimate \eqref{carl} holds uniformly in $\beta$ whenever  the following hold\! $:$ 
\vspace{-6pt}
\begin{enumerate} 
[leftmargin=0.8cm, labelsep=0.3 cm, topsep=0pt]
\item[$\bullet$]  $\gamma<1$, $0\le \frac1r-\frac1s\le 1-\gamma$, and  $(\frac1r,\frac1s)\neq (1,\gamma),$ $(1-\gamma,0)$.
\item[$\bullet$]   $\gamma=1$, $a=b=2$,  and $1<r=s<\infty$. 
\end{enumerate}
\end{lem}

Lemma \ref{carl_pf}   was implicit  in \cite{ev01}  with    the Lebesgue spaces instead of the Lorentz spaces. 
The extra condition $a=b=2$ when  $\gamma=1$ is due to limitation of the real interpolation in mixed norm spaces.   Once we have Lemma \ref{carl_pf},  the proof of Theorem \ref{thm_carl} is rather simple. 

 \begin{proof}[Proof of Theorem \ref{thm_carl}]  Let $\ppq$ be in $\mathfrak T$.  
 By real interpolation between the estimates in Corollary \ref{co_uniform} and inclusion relations between Lorentz spaces, we get \eqref{eq_uni} with $\gamma=\frac{d}2(\frac1p-\frac1q)$ for any $ 1\le a\le b\le\infty$ if  $p\neq 2$  
 and $q\neq 2$. Thus Lemma \ref{carl_pf} gives the estimate \eqref{carl} in the Lorentz spaces if  the exponents satisfy the condition in Theorem \ref{thm_carl}. 
\end{proof}

The estimate \eqref{carl}  is  equivalent to the Sobolev type  inequality
\begin{equation}\label{recarl1}
\| h\|_{L^s(\R; L_x^{q,b})}\le C \| (\Delta -|x|^2 +\partial_t + 2\beta +d)h\|_{L^r(\R; L_x^{p,a})},\quad h\in \mathrm C_c^\infty(\R^{d+1}).
\end{equation}
One can easily see this by following the argument in \cite{e00}.  
 Especially, if $r=s$, the inequality \eqref{recarl1} implies 
$ \|f\|_{q}\le C\| (\Delta-|x|^2 +2\beta +d) f\|_{p}$ for $f\in \mathrm C_c^\infty(\R^d)$ which is, in fact, a special case of \eqref{eq_resol} where $z=2\beta+d\not\in 2\mathbb N_0+d$.
Indeed, let $f_1$ be a compactly supported smooth function on $\R$ with $f_1(0)=1$. Then, the above estimate  follows by applying \eqref{recarl1} to the function $h(x,t)=f(x)f_1(t/R)R^{-1/r}$, $R>1$  and letting $R \to \infty$.

\begin{rem}\label{failure} When $r=s$, the implication    from \eqref{recarl1} to   \eqref{eq_resol} with $z=2\beta+d\not\in 2\mathbb N_0+d$    can be used to show that the Carleman estimate \eqref{carl} holds only if 
\[ 
 \frac1p-\frac1q\le \frac2d\,. 
 \]
 By the Marcinkiewicz multiplier theorem for the Hermite operator $H$ $($\cite[Theorem 4.2.1]{Th93}$)$
$(H-z)^{-1}H$ with $z=2\beta+d\not\in 2\mathbb N_0+d$ is bounded on $L^p$, $1<p<\infty$. Thus, we see that the Carleman estimate  \eqref{carl} implies  the estimate $
  \|H^{-1}u\|_{q}\lesssim \|u\|_{p}$ for $ u\in \mathrm C_c^\infty(\R^d). $
By Theorem \ref{sob}  the inequality  holds only if $1/p-1/q\le 2/d$.
\end{rem}

\begin{proof}[Proof of Lemma \ref{carl_pf}]
To  prove  Lemma \ref{carl_pf} we basically  rely on the argument in \cite{e00, ev01}, so  we shall be brief. By scaling,  it is easy to see  that 
\eqref{carl}  is equivalent to \eqref{recarl1}. See \cite{e00} for the details. Thus, 
we  need to show \eqref{recarl1}  by replacing  $h$ with   $(\Delta -|x|^2 +\partial_t + 2\beta +d)^{-1} g$. Applying the projection operator $\Pi_\lambda$  in $x$-variables and taking Fourier transform in $t$,  we  see  the operator $S_\beta:=(\Delta -|x|^2 +\partial_t + 2\beta +d)^{-1}$ is given by 
 \[S_\beta g(x,t) =\int_\R K_\beta(t-s)(g(\cdot, s))(x)ds,  \]
where the operator valued kernel $K_\beta$ is given   by 
\[ K_\beta(t)(f) =\frac12\int_\R e^{2\pi i t\tau} \sum_{k=0}^\infty \frac{\Pi_k(f)}{\pi i \tau +\beta -k}\, d\tau,\quad f\in \mathrm C_c^\infty(\R^d).\]
To prove \eqref{carl}, it is enough to show 
\begin{equation}\label{goal}
 \|S_\beta g\|_{L^s(\R; L_x^{q,b})}\lesssim \|g\|_{L^r(\R; L_x^{p,a})} ,\quad g\in \mathrm C_c^\infty(\R^{d+1})
\end{equation}
with implicit constant independent of $\beta$ as long as $\dist(\beta, \mathbb N_0)\ge c$ for some $c>0$.

We regard  $S_\beta$ as a  vector valued convolution operator. Let us first consider the case $\gamma<1$ which is easier. 
Let 
$\phi\in C^\infty_c([-1,1])$ such that $\phi(t)=1$ on $[-1/2,1/2]$. Breaking the integral 
with functions $\phi(t\tau), 1-\phi(t\tau)$ and using  integration by parts and \eqref{eq_uni}, it is easy to see that
$\|K_\beta (t)\|_{L_x^{p,a}\to L_x^{q,b}} \lesssim \min\{|t|^{-\gamma},|t|^{-2}\}.$
Since  $\gamma< 1$,    for $r, s$ satisfying $0\le \frac1r-\frac1s\le 1-\gamma$  and $(\frac1r,\frac1s)\neq (1,\gamma),$ $(1-\gamma,0)$ we obtain \eqref{goal} by Young's convolution inequality and the Hardy-Littlewood-Sobolev inequality.

We now turn to the case  $\gamma=1$. We claim that  the kernel $K_\beta$ satisfies the H\"ormander condition
\begin{equation}\label{horm}
 \sup_{s\neq 0}\int_{|t|>2|s|} \| K_\beta(t-s) -K_\beta(t)\|_{L^{p,2}\to L^{q,2}} \,dt \le A<\infty,\end{equation}
where $A$ is depending only on the constant  $c>0$ such that $\dist(\beta, \mathbb N_0)\ge c.$  To show \eqref{horm} it is sufficient to show $\|K_\beta'(t)\|_{L^{p,2} \to L^{q,2}}\lesssim |t|^{-2}$.\footnote{If $\|K_\beta'(t)\|_{L_x^{p,2}\to L_x^{q,2}} \lesssim |t|^{-2}$, 
$\|K_\beta(t-s)-K_\beta(t)\|_{L_x^{p,2}\to L_x^{q,2}} =\|\int_t^{t-s}  K_\beta'(\sigma) d\sigma \|_{L_x^{p,2}\to L_x^{q,2}} \lesssim |s||t|^{-2}.$
This clearly yields  \eqref{horm}.} By integration by parts we have 
\begin{align*}
 (-2\pi i t)^2 K_\beta'(t) 
 &=  2^2(\pi i)^3 \int_{-\infty}^\infty  \tau e^{2\pi i \tau t}\sum_{k=0}^\infty\frac1{(\pi\tau i +\beta -k)^3}\Pi_k \,d\tau.
 \end{align*} 
 The assumption \eqref{eq_uni} (with $\gamma=1$ and $m=3$) gives
$\||t|^2 K_\beta'(t)\|_{L^{p,2} \to L^{q,2}}\lesssim  1$
uniformly in $t$ and $\beta$ satisfying $\dist(\beta, \mathbb N_0)\ge c,$ which proves the claim  \eqref{horm}. 
Thanks to \eqref{horm} and the usual vector valued singular integral theory, in order to prove \eqref{goal} for $1<r=s<\infty$, it suffices to obtain the estimate \eqref{goal} with  $r=s=2$ and $a=b=2$.

For $\eta\in \mathrm C_c^\infty(\mathbb R)$ we define $\eta(D_t)$ by  $\mathcal F_t(\eta(D_t) g)(x,\tau) =\eta(\tau)\mathcal F_t {g}(x,\tau)$ where $\mathcal F_t$ denotes the Fourier transform in $t$. We use the following Littlewood-Paley type inequality in the Lorentz spaces. 

\begin{lem}\label{LPm} Let  $1<p,r<\infty$.  Suppose $\eta$ is a smooth function supported in $[2^{-2},1]$ which  satisfies $\sum_{j=-\infty}^\infty |\eta(2^{-j}t)|^2\sim 1$ for all $t>0$.  Then we have
\begin{equation}\label{eq:lpm}
\|g\|_{L^r_t(\R; L^{p,r}_x)} \lesssim  \| (\sum_{j\in\Z}|\eta(2^{-j}|D_t|)g|^2)^{1/2}\|_{L^r_t(\R; L_x^{p,r})} \lesssim \|g\|_{L^r_t(\R; L_x^{p,r})}.
 \end{equation}
\end{lem} 
\begin{proof}
It is sufficient to show the second inequality in \eqref{eq:lpm} because the first inequality follows from the second via the standard polarization argument and duality. 
For any $1<p,r<\infty$ we  have   $\| (\sum_{j\in\Z} |\eta(2^{-j}|D_t|)g|^2)^{1/2}\|_{L^r(\R; L^{p}(\R^d))} \lesssim \|g\|_{L^r(\R; L^{p}(\R^d))} $  
by means of  the usual Littlewood-Paley inequality and the vector valued singular integral theorem (see \cite[Lemma 2.1]{ev01}). We interpolate these estimates using 
 the real interpolation in the mixed-norm spaces, especially,  
\[ (L^{p_0}(\R; L^{q_0}), L^{p_1}(\R; L^{q_1}))_{\theta, p} =L^{p}(\R ; L^{q,p})\] 
whenever $p_0,q_0,p_1,q_1\in[1,\infty)$  and $\ppq=(1-\theta) (1/{p_0},1/{q_0}) =\theta(1/{p_1},1/{q_1})$ with $\theta\in (0,1)$ (see \cite{cw,LP}). Therefore,   we obtain the second inequality in \eqref{eq:lpm}. 
\end{proof}

We now note that 
$\psi(2^{-j}|D_t|) S_\beta g(x,t) =\int_{\R} K_{\beta,j}(t-s)g(\cdot, s)(x) dt$ where
\[ K_{\beta,j}(t)f(x): =\frac12\int_\R e^{2\pi it\tau} \psi\Big(\frac{|\tau|}{2^{j}}\Big)\sum_{k=0}^\infty \frac{1}{\pi i \tau +\beta-k}\Pi_kf(x) d\tau\,.\]
Using \eqref{eq_uni} with $a=b=2$ and integration by parts, we note that  $\|K_{\beta, j}(t)\|_{L_x^{p,2}\to L_x^{q,2}} $ $ \le C2^j(1+2^{j}|t|)^{-2}$ with $C$ independent of $j$ and $\beta$ if $\dist(\beta, \mathbb N_0)\ge c>0$. Thus, Young's convolution inequality gives
\begin{equation}\label{tem_QS}
\|\psi(2^{-j}|D_t|) S_\beta g\|_{L^2(\R;L_x^{q,2})}\lesssim \|g\|_{L^2(\R; L_x^{p,2})} 
\end{equation}
 with the implicit constant independent of  $j$ and $\beta$. To get  the desired \eqref{goal} with  $r=s=2$, we combine this inequality and Lemma \ref{LPm}. 
Since $2\le q<\infty$, the space $L^{( q/2),(2/2)}$ is normable.  So, 
\Be
\label{easy}
 \| (\sum_j |h_j|^2)^{1/2}\|_{L_x^{q,2}}\lesssim (\sum_j \| h_j\|_{L_x^{q,2}}^2)^{1/2} 
 \Ee
 Since $S_\beta g=\sum_{j\in\mathbb Z}\psi(2^{-j}|D_t|) S_\beta g$, applying Lemma \ref{LPm} and then \eqref{easy}, we have 
\begin{align*}
\|S_\beta g\|_{L^2(\R;L_x^{q,2})}
\lesssim  (\sum_{j\in\mathbb Z}\|\psi(2^{-j}|D_t|) S_\beta g \|_{L^2(\R;L_x^{q,2})}^2)^{\frac12}.
\end{align*}
Let $\widetilde \psi\in C_c([2^{-2}, 1])$ such that $\psi\widetilde \psi=\psi$, so $\psi(2^{-j}|D_t|) S_\beta g=\psi(2^{-j}|D_t|) S_\beta \widetilde\psi(2^{-j}|D_t|) g$. 
Using \eqref{tem_QS} followed by \eqref{eq:lpm}, we get 
\[\|S_\beta g\|_{L^2(\R;L_x^{q,2})} \lesssim   (\sum_{j\in\mathbb Z}\|\widetilde \psi(2^{-j}|D_t|) g \|_{L^2(\R;L_x^{p,2})}^2)^{1/2}.
\] 
By duality the inequality \eqref{easy} is equivalent to $ (\sum_j \| h_j\|_{L_x^{p,2}}^2)^{1/2} \lesssim \| (\sum_j |h_j|^2)^{1/2}\|_{L_x^{p,2}}$ for $1<p\le2$.
Thus, using Lemma  \ref{LPm} we get 
\[\|S_\beta g\|_{L^2(\R;L_x^{q,2})} \lesssim   \|(\sum_{j\in\mathbb Z} |\widetilde \psi(2^{-j}|D_t|) g|^2)^{1/2} \|_{L^2(\R;L_x^{p,2})}
\lesssim  \|g \|_{L^2(\R;L_x^{p,2})}.
\]   
This completes the proof. 
\end{proof}

\subsection*{Acknowledgements}
This work was  supported by the POSCO Science Fellowship and Grant no. NRF-2020R1F1A1A01048520 (E. Jeong) and Grant  no. NRF-2021R1A2B5B02001786 (S. Lee and J. Ryu).

\end{document}